\documentclass[12pt,oneside,a4paper]{amsart}

\usepackage{silence}
\usepackage{subfiles}
\usepackage[T1]{fontenc}
\usepackage[utf8]{inputenc}
\usepackage[dvipsnames]{xcolor}
\usepackage[yyyymmdd,hhmmss]{datetime}
\usepackage[numbers,sort,compress]{natbib}
\usepackage{amstext,amsthm,amssymb}
\usepackage{xparse,mathrsfs,mathtools}
\usepackage[british]{babel}
\usepackage[babel=true,english=british]{csquotes}
\usepackage{xr-hyper}
\usepackage{stmaryrd}
\usepackage{enumitem}
\usepackage[pdftex,unicode=true,%
pdfusetitle,pdfdisplaydoctitle=true,%
pdfpagemode=UseOutlines,%
bookmarks=true,bookmarksnumbered=false,bookmarksopen=true,bookmarksopenlevel=2,%
breaklinks=true,%
pdfencoding=unicode,psdextra,
pdfcreator={},
pdfborder={0 0 0},
colorlinks=false
]{hyperref}
\usepackage{lmodern,microtype}
\usepackage{geometry}
\usepackage[capitalise,nameinlink]{cleveref}
\usepackage[cal=zapfc,bb=ams,frak=esstix,scr=boondoxo]{mathalfa}

\hypersetup{
	colorlinks=false,
	pdfborder={0 0 0},
	pdfborderstyle={/S/U/W 0},
}

\newtheorem*{thm*}{Theorem}
\newtheorem{thm}{Theorem}[section]
\newtheorem{lem}[thm]{Lemma}

\newtheorem{prop}[thm]{Proposition}

\theoremstyle{remark}
\newtheorem{rem}[thm]{Remark}
\newtheorem*{rem*}{Remark}

\numberwithin{equation}{section}
\crefformat{equation}{#2(#1)#3}
\crefformat{enumi}{#2(#1)#3}
\crefname{section}{\textsection}{\textsection\textsection}
\newcommand{\ZZ}{\mathbb{Z}}

\newcommand{\RR}{\mathbb{R}}

\newcommand{\LandauO}{O}

\DeclarePairedDelimiter\parentheses{\lparen}{\rparen}
\DeclarePairedDelimiter\braces{\lbrace}{\rbrace}
\DeclarePairedDelimiter\brackets{\lbrack}{\rbrack}
\DeclarePairedDelimiter\abs{\lvert}{\rvert}
\DeclarePairedDelimiter\norm{\lVert}{\rVert}
\DeclarePairedDelimiter\floor{\lfloor}{\rfloor}

\NewDocumentCommand\set{ s o m o }{%
	\IfBooleanTF{#1}{\IfNoValueTF{#4}{\braces*{#3}}{\braces*{\,#3:#4\,}}}{%
		\IfNoValueTF{#2}{\IfNoValueTF{#4}{\braces{#3}}{\braces{\,#3:#4\,}}}{%
			\IfNoValueTF{#4}{\braces[#2]{#3}}{\braces[#2]{\,#3:#4\,}}}}%
}
\NewDocumentCommand\e{ s O{} m }{%
	\IfBooleanTF{#1}{%
		\operatorname{e}_{#2}\parentheses*{#3}%
	}{\operatorname{e}_{#2}\parentheses{#3}}%
}

\makeatother

\title[Discrete $\Omega$-results]{Discrete $\Omega$-results for the Riemann zeta function}



\author{Paolo~Minelli}
\address{%
	Paolo~Minelli\\%
	Department of Mathematics\\%
	KTH Royal Institute of Technology\\%
	Stockholm\\%
	Sweden
}
\email{pminelli@kth.se}

\author{Athanasios~Sourmelidis}
\address{%
	Athanasios~Sourmelidis\\%
	Insitute for Analysis and Number Theory\\%
	Graz University of Technology\\%
	Steyrergasse 30\\%
	8010~Graz\\%
	Austria%
}
\email{sourmelidis@math.tugraz.at}

\thanks{%
	PM thanks  the Knut and Alice Wallenberg foundation grant KAW 2019.0503 as well as the Austrian Science Fund (FWF), project I-3466 for their support.
	AS was supported by the Austrian Science Fund (FWF) project~M~3246-N}

\begin{document}
	\begin{abstract}
		We study lower bounds for the Riemann zeta function $\zeta(s)$ along vertical arithmetic progressions in the right-half of the critical strip.
		We show that the lower bounds obtained in the discrete case coincide, up to the constants in the exponential, with the ones known for the continuous case, that is when the imaginary part of $s$ ranges on a given interval.
		Our methods are based on a discretization of the resonance method for estimating extremal values of $\zeta(s)$. 
	\end{abstract}
	\maketitle
	
	\section{Introduction}
	\subsection{Main results}
	The aim of this paper is to provide lower bounds, also known as $\Omega$-results, of the Riemann zeta function $\zeta(s)$, $s:=\sigma+it$, along arithmetic progressions.
	More precisely, given $\alpha>0$ and $\beta\in [0, \alpha)$ we seek for lower bounds for the expression
	\begin{align*}
		\max_{\ell \in [T^\theta, T]\cap\mathbb{Z}}\vert \zeta(\sigma+i(\alpha\ell+\beta))\vert,
	\end{align*}
	whenever $\sigma \in [1/2, 1]$ and $\theta$ is a parameter to be specified later.

	We begin by investigating large values of $\zeta(s)$ inside the strip $1/2<\sigma< 1$, where we obtain discrete $\Omega$-results matching, up to the constant in the exponential, the known continuous lower bounds.
	From now on, we will tacitly assume that $\ell\in\mathbb{Z}$.
	\begin{thm}\label{thm:thm-strip}\label{thm:ourresult1}
		Let $1/2<\sigma<1$, $1/(2(1+\sigma))<\delta<1/2$, $\alpha>0$ and $\beta\in[0,\alpha)$ be fixed.
		Then for any sufficiently large $N\geq1$
		\[
		\max_{N\leq\ell\leq 2N}|\zeta(\sigma+i(\alpha\ell+\beta))|\geq \exp\left(\frac{(1+o(1))\parentheses*{(\frac{1}{2}-\delta)\log N}^{1-\sigma}}{(1-\sigma)(\log\log N)^\sigma}\right),
		\]
		where the implicit constant depends on the aforementioned parameters.
	\end{thm}
	
	We then restrict our attention to homogeneous progressions, i.e. $\beta=0$. 
	In this case we are able to establish a lower bound which matches, up to the constant in the exponential, the best lower bounds known for the continuous case. 
	
	\begin{thm}\label{thm:ourresult2}
		Let $\alpha>0$ be fixed.
		Then for any sufficiently large $N\geq1$
		\begin{align*}
			\max_{\sqrt{N}\leq\ell\leq N\log N}\abs*{ \zeta\parentheses*{\frac{1}{2}+i\alpha \ell}} \geq \exp\left(c\sqrt{\frac{\log N\log\log\log N}{\log\log N}}\right),
		\end{align*}
		where $c<1/\sqrt{2}$ is admissible.
	\end{thm}
	
	\begin{thm}\label{thm:ourresult3}
		Let $\alpha>0$ be fixed.
		Then for any sufficiently large $N\geq1$
		\begin{align*}
			\max_{\sqrt{N}\leq\ell\leq N}\abs*{ \zeta\parentheses*{1+i\alpha \ell}} \geq e^\gamma (\log\log N+\log\log\log N+\LandauO(1)).
		\end{align*}
		
	\end{thm}
	\subsection{Background on \texorpdfstring{$\Omega$-results}{O-results}}
	
	The investigation of $\Omega$-results for the Riemann zeta function $\zeta(s)$, $s:=\sigma+it$, as well as for other $L$-functions has a long history. 
	These questions originate from the converse investigations on $O$-results of $\zeta(s)$ in the critical strip $0<\sigma<1$, which pose very challenging problems.
	On the other hand, $\Omega$-results are obtained more successfully than $O$-results because one needs to verify their truth for some indefinitely large values of $t$, while an $O$-result is a statement about all large values of $t$. 
	Another reason for studying $\Omega$-results is that under the Riemann hypothesis, they most likely represent the real growth of $\zeta(s)$ when $\sigma>1/2$ (see for example \cite[Theorem 14.2]{Titzeta} and a heuristic argument in \cite[page 513]{Montgomeryzeta}).
	Bohr, Landau and Titchmarsh were among the first to develop methods in obtaining $\Omega$-results and we refer to \cite[Chapter VIII]{Titzeta} for an exposition.
	
	In 1977, Montgomery \cite{Montgomeryzeta} showed that for $\sigma\in(1/2,1)$ and $\theta=(\sigma-1/2)/3$,
	\begin{align}\label{montgomery}
		\max_{ t\in[T^\theta, T]}|\zeta(\sigma+it)|\geq\exp\parentheses*{C_\sigma\frac{(\log T)^{1-\sigma}}{(\log\log T)^\sigma}},
	\end{align}
	where $C_\sigma=(\sigma-1/2)^{1/2}/20$ is admissible.
	Montgomery also extended \cref{montgomery} to $\sigma=1/2$ and a suitable constant $C_{1/2}$ conditionally under the Riemann hypothesis.
	
	In recent years, Soundararajan \cite{Sound-res}  developed a new method called \textit{the resonance method}, which turned out to be particularly successful and flexible. 
	Roughly speaking, this method consists in comparing the averages
	\begin{align*}
		\int_\mathbb{R}  \zeta\parentheses*{\frac{1}{2}+it} \vert R(t)\vert^2 \Phi\parentheses*{\frac{t}{T}} dt\quad\text{ and }\quad\int_\mathbb{R} \vert R(t)\vert^2 \Phi\parentheses*{\frac{t}{T}} dt
	\end{align*}
	for a suitably selected Dirichlet polynomial $R(t)$, called {\it the resonator}, and a smooth function  $\Phi(x)$ with compact support.
	Using this method and maximising the quadratic form stemming out of the ratio of the above quantities, Soundararajan obtained \cref{montgomery} for $\sigma=1/2$ unconditionally and in a shorter range of $t$.
	\begin{thm*}[Soundararajan]
		If $T$ is sufficiently large, then
		\begin{align*}
			\max_{t\in [T, 2T]}\abs*{\zeta\parentheses*{\frac{1}{2}+it}}\geq \exp\left((1+o(1))\sqrt{\frac{\log T}{\log\log T}}\right).
		\end{align*}
	\end{thm*}
	
	Further developments of the \textit{resonance method} in connection with extremal problems involving \textit{gcd-sums} were carried out by Hilberdink \cite{Hilberdink-res} who showed that there is an intimate connection between extreme values of $\zeta(s)$ in the strip $1/2<\sigma<1$ and the gcd-sums.
	We refer to the introductory sections in \cite{Radziwill-Lewko} and the references therein for a quick survey.
	Incidentally, similar results and the development of the resonance method inside the strip $1/2<\sigma<1$ have been carried out earlier -yet passed unknown for decades- by Voronin (see \cite{voronin} or \cite[Chapter VIII, \S2]{Karatsuba}).
	However, in the works of both Hilberdink and Voronin, the double logarithm in the denominator on the right-hand side in \cref{montgomery} had exponent $1$ instead of $\sigma$. 
	It was first Aistleitner  \cite{Chris-res} who fruitfully employed the resonance method in connection with gcd-sums by additionally introducing a suitable {\it long resonator} $R(t)$ and by exploiting the positivity of the coefficients of the Dirichlet series expansion of $\zeta(s)$.
	
	\begin{thm*}[Aistleitner]
		Let $\sigma\in (1/2, 1)$ be fixed. 
		Then \cref{montgomery} holds with $\theta=1-\sigma$ and
		$
		C_\sigma=0.18(2\sigma-1)^{1-\sigma}
		$
		being admissible.
	\end{thm*}
	
	Shortly after, Bondarenko and Seip \cite{Bondarenko-Seip-1, Bondarenko-Seip-2, Bonda-Seip-survey} adopted ideas from \cite{Chris-res} and obtained the following improvement over Soundararajan's result.
	\begin{thm*}[Bondarenko and Seip, \cite{Bondarenko-Seip-2}] \label{thm:bondaseip}
		Let $\theta\in [0,1)$ and $c<\sqrt{1-\theta}$.
		Then
		\begin{align*}
			\max_{t\in [T^\theta, T]}\abs*{\zeta\parentheses*{\frac{1}{2}+it}}\geq \exp\left( c\sqrt{\frac{\log T\log \log\log T}{\log\log T}}\right)
		\end{align*}
		for any sufficiently large $T>0$.
	\end{thm*}
	The above result has been improved not long after by de la Bret\`eche and Tenenbaum 
	\cite{Bret-Tenenbaum}, who exploited the full contribution coming from gcd-sums and obtained that any positive constant $c<\sqrt{2(1-\theta)}$ is admissible.

	The case of $\sigma=1$ was considered first by Littlewood \cite[Theorem 8]{Littlewood}, who showed that
	there exist arbitrarily large $t>0$ such that
	$ 
	|\zeta(1+it)|\geq(1+o(1))e^\gamma\log\log t.
	$
	Levinson \cite{levinson} showed that on the right-hand side we can have $e^\gamma\log\log t+O(1)$, while Granville and Soundararajan \cite[Theorem 2]{granvillesound} obtained that
	\begin{align}\label{gransound}
		\max_{t\in[0,T]}|\zeta(1+it)|\geq e^\gamma(\log_2T+\log_3T-\log_4T+O(1))
	\end{align}
	for any sufficiently large $T>0$, where $\log_{j+1}:=\log_j$, $j\geq1$.
	It is worth mentioning here that their approach resembles Montgomery's method for $1/2<\sigma<1$ and it yields in particular discrete $\Omega$-results of $\zeta(s)$ on the line $1+i\mathbb{R}$ along homogeneous arithmetic progressions.
	Recently, Aistleitner, Mahatab and Munsch \cite[Theorem 1]{Chris-Marc-Kama} proved  that $-\log_4 T$ in \cref{gransound} can be omitted.

	\subsection{Discrete \texorpdfstring{$\Omega$-results}{O-results}}
	Our investigations are motivated by the following theorem of Li and Radziwill \cite[Theorem 5]{LiRad}, who modified Soundararajan's method to obtain
	\begin{thm*}[Li and Radziwill]
		Let $\alpha>0$ and $\beta\in \RR$. 
		Then for infinitely many integers $\ell>0$
		\begin{align}\label{lirad}
			\abs*{\zeta\parentheses*{\frac{1}{2}+i(\alpha\ell+\beta)}} \gg \exp\left((1+o(1))\sqrt{\frac{\log \ell}{6\log\log \ell}}\right).
		\end{align}
	\end{thm*}
	
	In \cref{thm:ourresult1} we obtain discrete lower bounds for $\zeta(s)$ in the strip $1/2<\sigma<1$ where, compared to Montgomery's and Aistleitner's results, we also improve upon the constant in the exponential and the interval where the discrete variable $\ell$ ranges over.
	Our method is based on a {\it discretization} of the resonance method combined with several ideas from \cite{Chris-Marc-Kama}, \cite{Bonda-Seip-survey} and \cite{LiRad}.
	On the other hand, in the continuous case there already exist superior results \cite[Theorem 1]{Bonda-Seip-survey}.
	
	We then turn to the case $\sigma=1/2$.
	\cref{thm:ourresult2} can be regarded as a partial improvement of Li and Radziwill's result.
	The restriction to homogeneous progressions is due to the nature of the resonating polynomial involved in the method of Bondarenko and Seip. 
	On the critical line, large values of $\zeta(s)$ are captured by significantly longer and sparser Dirichlet polynomials. 
	Dealing with these polynomials raises several issues in controlling the contribution of the non-constant Fourier coefficients introduced in our setting (see \cref{Remark} after the proof of \cref{thm:ourresult2}). 
	Another drawback of the long resonator method is that it yields $\Omega$-results in {\it long} intervals $[N^\theta,N]$, $0<\theta<1$, of real/integer numbers.

	Lastly, \cref{thm:ourresult3} is an improvement on the result of Granville and Soundararajan along homogeneous arithmetic progressions and matches the one of Aistleitner, Mahatab and Munsch.
	Its proof is a discrete variation of their method with only minor changes.
	Once more, the long resonator method restricts us to study $\Omega$-results in long intervals $[\sqrt{N},N]$.
	It is worth mentioning that Granville and Soundararajan are also forced to work in long intervals $[T^{1/10},T]$ in order to obtain the desired $\Omega$-result.
	
	For the sake of simplicity, we stated Theorem \ref{thm:ourresult2} and Theorem \ref{thm:ourresult3} for the intervals $[\sqrt{N},N\log N]$ and $[\sqrt{N},N]$, respectively.
	Their proofs can be modified in a straightforward manner to yield $\Omega$-results in any long intervals $[N^\theta,N]$, $0<\theta<1$, by changing the occurring constants $c$ and $O(1)$ to appropriately depend on $\theta$.
	Since our main goal was to show that discrete and continuous $\Omega$-results match up to constants in the exponential, we did not pursue this direction further.
	On a similar note, there is the possibility that if we were to build upon the method of de la Bret\`eche and Tenenbaum, then we could slightly improve upon the upper bound of the constant $c$ in \cref{thm:ourresult2}. 
	However, we decided to follow the strategy of Bondarenko and Seip, as we believe this choice avoids several complications and keeps the exposition more self-contained.
	
	For further results concerning the Riemann zeta function along vertical arithmetic progressions see also \cite{Frankenhui,Martinng,Putnam,Steuding}.

	\subsection{Notations and structure}
	The Fourier transform of a function $f\in L^1(\mathbb{R})$ will be denoted by 
	\begin{align*}
		\widehat{f}(\xi):=\int_\mathbb{R}f (x)e^{-2\pi i\xi x}\mathrm{d}x.
	\end{align*}
	
	A resonating polynomial will be a Dirichlet polynomial 
	$
	R(t):=\sum_{n\leq M}{r(n)}{n^{it}}
	$
	where the coefficients $r(n)\geq0$ and its length $M$ will be chosen differently each time depending on whether $\sigma=1/2$, $\sigma\in(1/2,1)$ or $\sigma=1$.
	Frequently, we will be using the relations
	\begin{align*}
		|R(t)|^2\leq R(0)^2\leq\#\mathcal{M}\sum_{n\leq M}r(n)^2=:\#\mathcal{M}\norm{R}_2^2,
	\end{align*}
	where $\mathcal{M}:=\set{n\in\mathbb{N}:r(n)\neq0}$.
	
	We use the Landau notation $f(x) = \LandauO(g(x))$ and the Vinogradov notation $f(x) \ll g(x)$ to mean that there exists some constant $C>0$ such that $\abs{f(x)} \leq C g(x)$ holds for all admissible values of $x$ (where the meaning of `admissible' will be clear from the context).
	Unless otherwise indicated, any dependence of $C$ on parameters other than $\alpha$, $\beta$, $\delta$ and $\sigma$ is specified using subscripts.
	Similarly, we write `$f(x) = o(g(x))$ as $x\to\infty$' if $g(x)$ is positive for all sufficiently large values of $x$ and $f(x)/g(x)$ tends to zero as $x\to\infty$.
	
	In \cref{Lemmas} we present some lemmas and propositions that are needed for the proofs of our results.
	The proof of \cref{Prop:reductiontogcd}, in particular, will be given in the last section, \cref{technicalprop}.
	In the remaining sections we give the proofs of the main theorems.
	
	\section{Auxiliary lemmas and resonating polynomials}\label{Lemmas}
	
	\begin{lem}[Sums of powers of divisors]\label{lem:gronwall}
		If $0<\sigma<1$, then for any sufficiently large integer $n>0$
		\[
		\mathop{\sum_{m\mid n}}\left(\frac{m}{n}\right)^\sigma\leq\exp\left(\frac{(1+o(1))\left(\log n\right)^{1-\sigma}}{(1-\sigma)\log\log n}\right).
		\]
	\end{lem}
	\begin{proof}
		For a proof see \cite[pages 119-122]{Gronwall}.
	\end{proof}
	
	\begin{lem}[Poisson summation formula]\label{lem:Poisson}
		Let $f\in C^2(\mathbb{R})$ be a complex-valued function with $f''\in L^1(\mathbb{R})$ and $f(x)\ll(1+|x|^2)^{-1}$, $x\in\mathbb{R}$.
		Then, for any $a\in\mathbb{R}$
		\[
		\sum_{\ell\in\mathbb{Z}}f(\ell+a)=\sum_{\ell\in\mathbb{Z}}\widehat{f}(\ell)e^{2\pi i\ell a}.
		\]
	\end{lem}
	\begin{proof}
		For a proof see \cite[Satz 2.2.2]{bruedern}.
	\end{proof}
	
	\begin{lem}[Gallagher]\label{gallagher}
		If $0<A<B-2$ and $f:[A,B]\to\mathbb{C}$ is continuously differentiable then
		\begin{align*}
			\sum_{A+1\leq n\leq B-1}|f(n)|^2\ll\int_{A}^B|f(t)|^2\mathrm{d}t+\parentheses*{\int_{A}^B|f(t)|^2\mathrm{d}t\int_{A}^B|f'(t)|^2\mathrm{d}t}^{1/2}.
		\end{align*}
	\end{lem}
	\begin{proof}
		For a proof see \cite[Lemma 1.4]{montgomerybook}.
	\end{proof}

	\begin{lem}\label{lem:discretemean}
		Let $\alpha>0$ be fixed.
		Then, for any sufficiently large $N\geq1$
		\begin{align*}
			\sum_{\ell\leq N}\abs*{\zeta\parentheses*{\frac{1}{2}+i\alpha\ell}}^2\ll N(\log N)^{3/2}.
		\end{align*}
		\begin{proof}
			As a matter of fact, the discrete mean square of $\zeta(s)$ on the critical line is expected to be  $\asymp N\log N$.
			This has been proved in \cite[Theorem 2]{LiRad} by introducing a weighted mean square of $\zeta(s)$ multiplied with a smooth and compactly supported function and recently in \cite[Theorem 1.1]{Koba} without a weight function and for any parameter $\alpha>0$ from a set of full Lebesgue measure.
			Since we will not have a weight function in our case and we wish to treat any $\alpha>0$, we provide a crude upper bound for the mean square, which, nevertheless, will suffice for the purpose we need it.
			We briefly sketch its proof here. 
			It is known by Hardy and Littlewood \cite[Theorem 7.3]{Titzeta} and Ingham \cite{ingham} that
			\begin{align*}
				\int_{0}^{T}|\zeta(1/2+it)|^2\mathrm{d}t\sim T\log T\quad\text{ and }\quad\int_{0}^{T}|\zeta'(1/2+it)|^2\mathrm{d}t\sim T(\log T)^2.
			\end{align*}
			The lemma follows now from these relations and application of Gallagher's lemma with $f(t)=\zeta(1/2+i\alpha t)$.
		\end{proof}
		
	\end{lem}
	\begin{lem}\label{lem: Bonda-Seip-gcd-bound}
		Let $x>0$ and $\ell\in\mathbb{N}$. 
		Assume that $\mathcal{M}$ denotes the set of divisors of 
		\[
		M=M(x,\ell):=\prod_{p\leq x}p^{\ell-1}
		\]
		and that $r(n)$, $n\in\mathbb{N}$, is the characteristic function of $\mathcal{M}$.
		If $0<\sigma<1$, then
		\[
		\sum_{mk=n\leq M}\frac{r(m)r(n)}{k^\sigma}\geq\sum_{n\leq M}r(n)^2\prod_{p\leq x}\left(1+p^{-\sigma}\right)^{1-1/\ell}.
		\]
		In particular, if we choose $x=(1/2-\delta)\log N/\log\log N$ for some $0<\delta<1/2$, and $\ell=\lfloor\log\log N\rfloor$, then $M\leq N^{1/2-\delta}$ and
		\[
		\sum_{mk=n\leq M}\frac{r(m)r(n)}{k^\sigma}\geq\sum_{n\leq M}r(n)^2\exp\left(\frac{(1+o(1))\parentheses*{\parentheses*{\frac{1}{2}-\delta}\log N}^{1-\sigma}}{(1-\sigma)(\log\log N)^\sigma}\right).
		\]
	\end{lem}
	\begin{proof}
		This follows \textit{mutatis mutandis} from Sections 4.2 and 4.3 of \cite{Bonda-Seip-survey}.
	\end{proof}
	
	In order to prove \cref{thm:ourresult1} we will need the following technical proposition whose proof is given in \cref{technicalprop}.
	\begin{prop}[Reduction to gcd-sums]\label{Prop:reductiontogcd}
		Let $N\in\mathbb{N}$, $1/2<\sigma<1$ and $\Phi(x)\geq0$ be a smooth and compactly supported function with support in $[1,2]$ and $\widehat{\Phi }(0)=1$. 
		If $R(t)$ is a Dirichlet polynomial with coefficients $r(n)\in\lbrace0,1\rbrace$ and length $M\leq N^{1/2-\delta}$ for some $1/(2(1+\sigma))<\delta<1/2$, then 
		\begin{align*}
			\sum_{\ell\in\mathbb{Z}}\sum_{k\leq 3\alpha N}\frac{|R(\alpha\ell+\beta)|^2}{k^{\sigma+i(\alpha\ell+\beta)}}\Phi \parentheses*{\frac{\ell}{N}}
			=N\sum_{mk=n\leq M}\frac{r(m)r(n)}{k^\sigma}+\mathcal{E}(N),
		\end{align*}
		where
		\begin{align*}
			|\mathcal{E}(N)|\leq N\norm{R}_2^2\exp\parentheses*{\frac{(1+o(1))\left(\log N\right)^{1-\sigma}}{(1-\sigma)\log\log N}}.
		\end{align*}
	\end{prop}
	
	In order to prove \cref{thm:ourresult2} we will employ the resonating polynomial constructed by Bondarenko and Seip in \cite[Section 3]{Bondarenko-Seip-2}.
	In the next proposition we collect several relations that are proved in that section and which we will also use in our proof.
	\begin{prop}\label{summaryofBondSeip}
		Let $\gamma,\kappa\in(0,1)$ be fixed and $\Phi(x):=e^{-x^2/2}$, $x\in\mathbb{R}$.
		Then, for any sufficiently large $T\geq1$, there is a Dirichlet polynomial $R(t):=\sum_{n\in \mathcal{M}}r(n)n^{-it}$ with $r(n)>0$ and being supported on a set of cardinality $0<\#\mathcal{M}\leq T^\kappa$ such that
		\begin{align*}
			\int_{\mathbb{R}}|R(t)|^2\Phi\parentheses*{\frac{t}{T}}\mathrm{d} t\asymp T\norm{R}_2^2,\qquad\int_{\mathbb{R}}|R'(t)|^2\Phi\parentheses*{\frac{t}{T}}\mathrm{d} t\ll T(\log T)^4\norm{R}_2^2
		\end{align*}
		and
		\begin{align*}
			\int_{\mathbb{R}}\sum_{n\leq \alpha T\log T}\frac{|R(t)|^2}{n^{1/2+it}}\Phi\parentheses*{\frac{t}{T}}\mathrm{d}t\gg T\exp\parentheses*{(\gamma+o(1))\sqrt{\kappa\frac{\log T\log\log\log T}{\log\log T}}}\norm{R}_2^2. 
		\end{align*}
		All of the above implicit constants depend on $\gamma$ and $\kappa$.
	\end{prop}
	\begin{proof}
		In \cite[Lemma 5, (17)-(20)]{Bondarenko-Seip-2} it is proved that
		\begin{align*}
			\int_{\mathbb{R}}|R(t)|^2\Phi\parentheses*{\frac{t}{T}}\mathrm{d} t=\sqrt{2\pi}T\sum_{n,\,m\in\mathcal{M}}r(n)r(m)\Phi\parentheses*{2\pi T\log\frac{n}{m}}\ll T\norm{R}_2^2,
		\end{align*}
		using the fact that $\widehat{\Phi}(x)=\sqrt{2\pi} \Phi(2\pi x)$, $x\in\mathbb{R}.$ 
		In particular, since all terms of the above double sum are positive, we can also obtain that
		\[
		\int_{\mathbb{R}}|R(t)|^2\Phi\parentheses*{\frac{t}{T}}\mathrm{d} t\geq\sqrt{2\pi}T\sum_{n\in\mathcal{M}}r(n)^2=\sqrt{2\pi}T\norm{R}_2^2.
		\]
		In a similar fashion
		\begin{align*}
			\int_{\mathbb{R}}|R'(t)|^2\Phi\parentheses*{\frac{t}{T}}\mathrm{d} t&=\sqrt{2\pi}T\sum_{n,\,m\in\mathcal{M}}r(n)r(m)\log n\log m\Phi\parentheses*{2\pi T\log\frac{n}{m}}\\
			&\ll T(\log\nu)^2\norm{R}_2^2,
		\end{align*}
		where $\nu:=\max_{n\in\mathcal{M}}n$.
		It is not difficult to see by the construction of the set $\mathcal{M}$ in the beginning of \cite[Section 3]{Bondarenko-Seip-2} that $\log \nu\ll(\log T)^2$.
		Indeed, if $N=\floor{T^\kappa}$ then any $n\in \mathcal{M}$ is a square-free number with at most $\frac{a\log N}{k^2\log_3 N}$ ($a$ depends on $\gamma$) primes in the range $[e^{k}\log N\log_2 N, e^{k+1}\log N\log_2 N]$ for any integer $k\in [1, (\log_2 N)^\gamma]$.
		Therefore, for any such $n$ we have the crude bound
		\begin{align*}
			n\leq\prod_{k\leq(\log_2 N)^\gamma} (e^{k+1}  \log N \log _2 N)^{\frac{a\log N}{k^2\log_3 N}}
		\end{align*}
		or
		\[
		\log n \ll \frac{\log N (\log_2 N)^2}{\log_3 N}.
		\]
		This yields the second relation of the proposition.
		Finally, the last claim of the proposition is an immediate consequence of \cite[Lemma 6]{Bondarenko-Seip-2} and \cite[(19)]{Bondarenko-Seip-2}.
		
	\end{proof}
	\section{Proof of \texorpdfstring{\cref{thm:ourresult1}}{Theorem 1.1}}
	\begin{proof}
		Let $\Phi(x)$ and $R(t)$ be as in \cref{Prop:reductiontogcd}.
		By the following approximation for the Riemann zeta function (see \cite[Theorem 4.11]{Titzeta})
		\begin{align}\label{eq:AFE}
			\zeta(\sigma+it)=\sum_{n\leq T}\frac{1}{n^{\sigma+it}}-\frac{T^{1-s}}{1-s}+\LandauO\parentheses*{\frac{1}{T^{\sigma}}},\quad1\leq|t|\leq T,\quad1/2\leq\sigma\leq1,
		\end{align}
		we deduce that
		\begin{align}\label{eq:firstreduction}
			\begin{split}
				\sum_{\ell\in\mathbb{Z}}\zeta(\sigma+i(\alpha\ell+\beta))|R(\alpha\ell+\beta)|^2\Phi\parentheses*{\frac{\ell}{N}}&=\sum_{\ell\in\mathbb{Z}}\sum_{k\leq 3\alpha N}\frac{|R(\alpha\ell+\beta)|^2}{k^{\sigma+i(\alpha\ell+\beta)}}\Phi\parentheses*{\frac{\ell}{N}}+\\
				&\quad+O\parentheses*{N^{-\sigma}\sum_{N\leq\ell\leq2N}|R(\alpha\ell+\beta)|^2}.
			\end{split}
		\end{align}
		Since $M< N^{1/2}$, the Montgomery-Vaughan mean value theorem (see \cite[Theorem 5.3]{Ivic}) yields 
		\begin{align*}
			\sum_{N\leq\ell\leq2N}|R(\alpha\ell+\beta)|^2=\sum_{N\leq \ell\leq 2N}\left|\sum_{n\leq M}\frac{r(n)n^{-i\beta}}{n^{i\alpha\ell}}\right|^2\ll N\log N\norm{R}_2^2
		\end{align*}
		and, therefore, 
		\begin{align}\label{last}
			\begin{split}
				\max_{N\leq\ell\leq2N}|\zeta(\sigma+i(\alpha\ell+\beta))|
				&\gg\frac{1}{N\log N\norm{R}_2^2}\sum_{\ell\in\mathbb{Z}}\zeta(\sigma+i(\alpha\ell+\beta))|R(\alpha\ell+\beta)|^2\Phi\parentheses*{\frac{\ell}{N}}.
			\end{split}
		\end{align}
		In view of \cref{eq:firstreduction} and \cref{Prop:reductiontogcd}, the right-hand side above is equal to
		\begin{align*}
			\frac{1}{N\log N\norm{R}_2^2}\parentheses*{N\sum_{mk=n\leq M}\frac{r(m)r(n)}{k^\sigma}+\mathcal{E}(N)}.
		\end{align*}
		If we choose now the length $M$ and the coefficients $r(n)$ of $R(t)$ as in \cref{lem: Bonda-Seip-gcd-bound}, we obtain that the above expression is greater than or equal to
		\begin{align*}
			\frac{1}{\log N}\brackets*{\exp\parentheses*{\frac{(1+o(1))\parentheses*{\parentheses*{\frac{1}{2}-\delta}\log N}^{1-\sigma}}{(1-\sigma)(\log\log N)^\sigma}}-\exp\parentheses*{\frac{(1+o(1))\parentheses*{\log N}^{1-\sigma}}{(1-\sigma)\log\log N}}}.
		\end{align*}
		The theorem now follows.
	\end{proof}

	\section{Proof of \texorpdfstring{\cref{thm:ourresult2}}{Theorem 1.2}}
	Let $\Phi(x)$ and $R(t)$ be as in \cref{summaryofBondSeip}.
	We define the quantities
	\begin{align*}
		D_{1, R}(N)=\sum_{\ell \in \ZZ} \vert R(\alpha \ell) \vert^2 \Phi \parentheses*{\frac{\ell}{N}}
		\quad\text{ and }\quad
		D_{2, R}(N)=\sum_{\ell \in \ZZ} \sum_{n\leq\alpha N\log N} \frac{\vert R(\alpha \ell)\vert^2}{n^{1/2+i\alpha\ell}} \Phi \parentheses*{\frac{\ell}{N}},
	\end{align*}
	which we estimate from above and below, respectively.
	
	We first use the exponential decay of $\Phi$ to bound $D_{1,R}(N)$ by a finite sum
	\begin{align*}
		D_{1, R}(N)\ll\sum_{\ell\leq N\log N}|R(\alpha\ell)|^2\Phi\parentheses*{\frac{\ell}{N}}+R(0)^2.	
	\end{align*}
	By Gallagher's lemma (\cref{gallagher}) with
	$
	f(t):=R(\alpha t)\sqrt{\Phi\left({t}/{N}\right)}
	$
	and an application of Minkowski's inequality it follows that $D_{1,R}(N)$ is bounded from above by
	\begin{align*}
		\int_\mathbb{R}|R(\alpha t)|^2\Phi\parentheses*{\frac{t}{N}}\mathrm{d}t
		&+\parentheses*{\int_\mathbb{R}|R(\alpha t)|^2\Phi\parentheses*{\frac{t}{N}}\mathrm{d}t\int_\mathbb{R}|R'(\alpha t)|^2 \Phi\parentheses*{\frac{t}{N}}\mathrm{d}t}^{1/2}+\\
		&+\parentheses*{\int_\mathbb{R}|R(\alpha t)|^2\Phi\parentheses*{\frac{t}{N}}\mathrm{d}t\int_\mathbb{R}|R(\alpha t)|^2 \Phi\parentheses*{\frac{t}{N}}\frac{t^2}{N^4}\mathrm{d}t}^{1/2}+R(0)^2.
	\end{align*}
	Hence, \cref{summaryofBondSeip} yields that
	\begin{align}\label{eq:D1R}
		\begin{split}
			D_{1,R}(N)&\ll N(\log N)^2\norm{R}_2^2+\norm{R}_2R(0)\parentheses*{\int_{\mathbb{R}}\Phi(y)y^2\mathrm{d}y}^{1/2}+R(0)^2\\
			&\ll N(\log N)^2\norm{R}_2^2.
		\end{split}
	\end{align}
	
	In the case of $D_{2,R}(N)$ it follows by squaring out, interchanging summations and applying the Poisson summation formula in \cref{lem:Poisson} that
	\begin{align*}
		D_{2, R}(N)=N\sum_{\substack{n\leq \alpha N\log N\\ h,\,k \in \mathcal{M}}} \frac{r(h)r(k)}{\sqrt{n}}\sum_{\ell\in\ZZ}\widehat{\Phi}\parentheses*{N\parentheses*{\frac{\alpha}{2\pi}\log\frac{k}{hn}-\ell}}.
	\end{align*}
	Due to the positivity of $\widehat{\Phi}(x)$ we can discard all terms of the inner sum corresponding to $\ell\neq0$ and, in view of \cref{summaryofBondSeip}, we obtain that 
	\begin{align*}
		\begin{split}
			D_{2, R}(N)&\geq N\sum_{\substack{n\leq \alpha N\log N\\ h,\,k \in \mathcal{M}}} \frac{r(h)r(k)}{\sqrt{n}}\widehat{\Phi}\parentheses*{N\frac{\alpha}{2\pi}\log\frac{k}{hn}}\\
			&=\int_\mathbb{R}\sum_{n\leq\alpha N\log N}\frac{\vert R(\alpha t)\vert^2}{n^{1/2+i\alpha t}}\Phi\parentheses*{\frac{ t}{T}} dt\\
			&\geq N\exp\parentheses*{(\gamma+o(1))\sqrt{\kappa\frac{\log N\log\log\log N}{\log\log N}}}\norm{R}_2^2.
		\end{split}
	\end{align*}
	Therefore,
	\begin{align}\label{eq:finalupperboundnumerator}
		\frac{D_{2,R}(N)}{D_{1,R}(N)}\geq\exp\parentheses*{(\gamma+o(1))\sqrt{\kappa\frac{\log N\log\log\log N}{\log\log N}}}.
	\end{align}
	From the approximation \cref{eq:AFE} for $\zeta(s)$ we deduce now that
	\begin{align*}
		\begin{split}
			\sum_{\sqrt{N}\leq|\ell|\leq N\log N} \zeta\parentheses*{\frac{1}{2}+i\alpha \ell} &\vert R(\alpha \ell) \vert^2 \Phi \parentheses*{\frac{\ell}{N}}-\sum_{\substack{\sqrt{N}\leq|\ell|\leq N\log N\\n\leq\alpha N\log N}}\frac{\vert R(\alpha \ell) \vert^2}{n^{1/2+i\alpha\ell}} \Phi \parentheses*{\frac{\ell}{N}}\\
			&\ll R(0)^2\sum_{\sqrt{N}\leq|\ell|\leq N\log N}\parentheses*{\frac{\sqrt{N\log N}}{\ell}+\frac{1}{\sqrt{N\log N}}}\\
			&\ll N^{1/2+\kappa}(\log N)^{3/2}\norm{R}_2^2.
		\end{split}
	\end{align*} 
	The rapid decay of $\Phi(x)$ yields that	
	\begin{align*}
		\sum_{\substack{|\ell|>N\log N\\n\leq\alpha N\log N}}\frac{|R(\alpha\ell)|^2}{n^{1/2+i\alpha\ell}}\Phi\parentheses*{\frac{\ell}{N}}=o(1)\norm{R}_2^2.
	\end{align*}
	On the other hand, the approximation \cref{eq:AFE} implies that
	\begin{align*}
		\sum_{\substack{|\ell|<\sqrt{N}\\n\leq\alpha N\log N}}\frac{\vert R(\alpha \ell) \vert^2}{n^{1/2+i\alpha\ell}} \Phi \parentheses*{\frac{\ell}{N}}
		&\leq R(0)^2\sum_{|\ell|<\sqrt{N}}\abs*{\sum_{n\leq \alpha N\log N}\frac{1}{n^{1/2+i\alpha\ell}}}\\
		&\leq\#\mathcal{M}\norm{R}_2^2\brackets*{\sum_{|\ell|<\sqrt{N}}\abs*{\zeta\parentheses*{\frac{1}{2}+i\alpha \ell}}+O\parentheses*{N^{1/2}(\log N)^{3/2}}}\\
		&\ll	N^{1/2+\kappa}(\log N)^{3/2}\norm{R}_2^2,
	\end{align*}
	where the last relation follows by applying the Cauchy-Schwarz inequality in the sum involving $\zeta(s)$ and by recalling \cref{lem:discretemean}.
	Therefore,
	\begin{align*}
		\sum_{\sqrt{N}\leq|\ell|\leq N\log N} \zeta\parentheses*{\frac{1}{2}+i\alpha \ell} \vert R(\alpha \ell) \vert^2 \Phi \parentheses*{\frac{\ell}{N}}=D_{2,R}(N)+\LandauO\parentheses*{N^{1/2+\kappa}(\log N)^{3/2}\norm{R}_2^2}.
	\end{align*}
	If we divide both sides with $D_{1,R}(N)$ then we obtain 
	\begin{align}\label{eq:finallyfinal}
		\max_{\sqrt{N}\leq\ell\leq N\log N}\abs*{\zeta\parentheses*{\frac{1}{2}+i\alpha\ell}}\geq\frac{D_{2,R}(N)}{D_{1,R}(N)}+\LandauO\parentheses*{\frac{N^{1/2+\kappa}(\log N)^{3/2}\norm{R}_2^2}{D_{1,R}(N)}}.
	\end{align}
	In view of \cref{eq:finalupperboundnumerator} it suffices to show that the error term above is $o(1)$.
	Observe that an application of Poisson summation on $D_{1,R}(N)$ implies that
	\[
	D_{1,R}(N)=N\sum_{h,\,k \in \mathcal{M}}{r(h)r(k)}\sum_{\ell\in\ZZ}\widehat{\Phi}\parentheses*{N\parentheses*{\frac{\alpha}{2\pi}\log\frac{k}{h}-\ell}}.
	\]
	Since all the terms in the double sum are positive, \cref{summaryofBondSeip} yields that
	\[
	D_{1,R}(N)\geq N\sum_{h,\,k \in \mathcal{M}}{r(h)r(k)}\widehat{\Phi}\parentheses*{N\parentheses*{\frac{\alpha}{2\pi}\log\frac{k}{h}}}=\int_{\mathbb{R}}|R(t)|^2{\Phi}\parentheses*{\frac{t}{N}}\mathrm d{t}\gg N\norm{R}_2^2.
	\]
	Thus, the error term in \cref{eq:finallyfinal} is $o(1)$ if we choose $\kappa<1/2$.\qed
	\begin{rem}\label{Remark}
		In obtaining \cref{eq:finalupperboundnumerator} we took advantage of the fact that $D_{2,R}(N)$ consists of only positive terms after applying Poisson's summation formula.
		If in its definition we had $\alpha\ell+\beta$ instead of $\alpha\ell$ (for some $\beta\neq0$), we could still apply Poisson.
		In this case however, the resulting sums would include a factor $(hn/k)^{-i\beta}$ for the various values of $h,k$ and $n$. 
		Clearly, the real parts of the terms in these sums do not have the same sign anymore, and so we do not have the freedom to discard terms from the sum on $\ell\in\mathbb{Z}$ while preserving a fixed sign for the real part of $D_{2,R}(N)$.
	\end{rem}
	
	\section{Proof of \texorpdfstring{\cref{thm:ourresult3}}{Theorem 1.3}}
	
	The setting is exactly the same as in \cite{Chris-Marc-Kama}.
	We let $\Phi(x)=e^{-x^2/2}$, $x\in\mathbb{R}$, and for every $N\geq1$ we define the resonating polynomial as
	\begin{align*}
		R(t):=\prod_{p\leq x}\parentheses*{1-q_pp^{it}}^{-1}=\sum_{n\geq1}q_nn^{it},\quad x:=\frac{\log N}{6\log\log N},
	\end{align*}
	where the coefficients $q_n$, $n\geq1$, are constructed in a completely multiplicative way with $q_p:=1-p/x$ for $p\leq x$ and $0$ for $p>x$.
	It can be confirmed \cite[(4)]{Chris-Marc-Kama} that $|R(t)|^2\leq N^{1/3+o(1)}$, $t\in\mathbb{R}$.
	
	We also define for any $y\geq2$ and $t\in\mathbb{R}$ the truncated Euler product by
	\begin{align*}
		\zeta(1+it;y):=\prod_{p\leq y}\parentheses*{1-p^{-1-it}}^{-1}:=\sum_{k\geq1}a_{k,y}k^{-it},
	\end{align*}
	where the coefficients satisfy $a_{k,y}\geq a_{k,x}\geq0$ for $x\leq y$.
	
	With the above settings, we will compare for any $y\geq x$ the quantities
	\begin{align*}
		G_{1, R}(N)=\sum_{\ell\in\mathbb{Z}}\zeta(1+i\alpha\ell; y) \abs*{ R(\alpha \ell)}^2 \Phi\parentheses*{\frac{\ell\log N}{N}}
	\end{align*}
	and
	\begin{align*}
		G_{2, R}(N)=\sum_{\ell\in\mathbb{Z}} \abs*{R(\alpha\ell)}^2 \Phi\parentheses*{\frac{\ell \log N}{N}}.
	\end{align*}
	We square out $\vert R(\alpha t)\vert^2$ and interchange summations since all sums are absolutely convergent. Then applying the Poisson summation formula in \cref{lem:Poisson} yields that
	\begin{align*}
		G_{1, R}(N)=\sum_{k\geq1}a_{k,y}\sum_{m,\,n\geq1}q_mq_n \sum_{\ell\in \ZZ} \int_\mathbb{R} \parentheses*{\frac{m}{kn}}^{-i\alpha t}\Phi\parentheses*{\frac{t\log N}{N} } e^{-2\pi i\ell t} dt=:\sum_{k\geq1} a_{k,y}I_k
	\end{align*}
	and
	\begin{align*}
		G_{2,R}(N)= \sum_{m,\,n\geq1} q_mq_n \sum_{\substack{\ell\in \ZZ}}\int_{\mathbb{R}} \parentheses*{\frac{m}{n}}^{-i\alpha t}\Phi\parentheses*{\frac{t\log N}{N} } e^{-2\pi i\ell t} \mathrm{d}t.
	\end{align*}
	Observe that the integrals above are non-negative real numbers by the positivity of the Fourier transform.
	Hence, 
	\begin{align}\label{eq:lowerboundD1R}
		G_{2,R}(N)\geq\sum_{n\geq1}q_n^2N\sum_{\ell\in\mathbb{Z}}\widehat{\Phi}(\ell N\log N)\geq Nq_1^2\widehat{\Phi}(0)=\sqrt{2\pi} N.
	\end{align}

	Fix now $k\geq1$.
	We rely again on the fact that all terms in $I_k$ are non-negative.
	Therefore, we can discard several terms and obtain by employing the completely multiplicative structure of $q_m$ that
	\begin{align*}
		\begin{split}
			I_{k}&\geq \sum_{n\geq1} \sum_{\substack{m\geq1\\ k\vert m}} q_mq_n \sum_{\ell\in \ZZ}\int_\mathbb{R} \parentheses*{\frac{m}{kn}}^{-i\alpha t}\Phi\parentheses*{\frac{t\log N}{N} } e^{-2\pi i\ell t} \mathrm{d}t\\
			&= q_k\sum_{n,\,r\geq1} q_nq_r \sum_{\ell\in \ZZ}\int_\mathbb{R} \parentheses*{\frac{r}{n}}^{-i\alpha t}\Phi\parentheses*{\frac{t\log N}{N} } e^{-2\pi i\ell t} \mathrm{d}t\\
			&=q_k G_{2,R}(N).
		\end{split}
	\end{align*}
	Thus
	\begin{align}\label{eq:goodratio}
		\frac{G_{1,R}(N)}{G_{2,R}(N)} \geq \sum_{k\geq1} a_{k,y} q_k\geq \sum_{k\geq1} a_{k,x} q_k\geq e^\gamma \left( \log\log N+\log\log\log N\right)+\LandauO(1),
	\end{align}
	where the last inequality has been already established in \cite[(10)--(11)]{Chris-Marc-Kama}.
	
	It remains to involve $\zeta(1+i\alpha\ell)$ in \cref{eq:goodratio}.
	In view of the approximation (see for example \cite[Lemma 1]{Chris-Marc-Kama})
	\begin{align*}
		\zeta(1+it)=\zeta(1+it;y)\parentheses*{1+O\parentheses*{\frac{1}{\log N}}},\quad y:=\exp\parentheses*{(\log N)^{11}},
	\end{align*}
	which is valid for $\alpha\sqrt{N}\leq|t|\leq \alpha N$ and any sufficiently large $N\geq1$, we deduce that
	\begin{align*}
		G_{2,R}(N)\max_{\sqrt{N}\leq\ell\leq N}|\zeta(1+i\alpha\ell)|&\geq\abs*{\sum_{\sqrt{N}\leq|\ell|\leq N}\zeta(1+i\alpha\ell;y)|R(\alpha\ell)|^2\Phi\parentheses*{\frac{\ell\log N}{N}}}+\\
		&\quad+O\parentheses*{\frac{G_{2,R}(N)}{\log N}}.
	\end{align*}
	Since $|R(t)|^2\leq N^{1/3+o(1)}$ and $\zeta(1+it;y)\ll\log y\ll(\log N)^{11}$, the rapid decay of $\Phi(x)$ implies that
	\begin{align*}
		\braces*{\sum_{|\ell|\leq\sqrt N}+\sum_{|\ell|\geq N}}\zeta(1+i\alpha\ell;y)|R(\alpha\ell)|^2\Phi\parentheses*{\frac{\ell\log N}{N}}=O\parentheses*{N^{5/6+o(1)}}+o(1).
	\end{align*}
	Therefore,
	\begin{align*}
		G_{2,R}(N)\max_{\sqrt{N}\leq\ell\leq N}|\zeta(1+i\alpha\ell)|\geq G_{1,R}(N)+O\parentheses*{N^{5/6+o(1)}+\frac{G_{2,R}(N)}{\log N}}
	\end{align*}
	and the theorem follows in view of \cref{eq:lowerboundD1R} and \cref{eq:goodratio}.
	\section{Proof of \texorpdfstring{\cref{Prop:reductiontogcd}}{Proposition 2.6}}\label{technicalprop}
	Before giving the proof of the proposition, which is inspired by \cite[Lemma 3]{LiRad}, we note that if $\Phi(x)$ is a smooth and compactly supported function in $\mathbb{R}$, then repeated partial integration yields that
	\[
	\widehat{\Phi}(\xi)\ll_A\parentheses*{\frac{1}{1+|\xi|}}^A,\quad\xi\in\mathbb{R},
	\]
	for any integer $A>0$.
	We will make use of this property frequently and any dependence of subsequent implicit constants on the parameter $A$, which will always be considered sufficiently large, is suppressed.
	\begin{proof}[Proof of \cref{Prop:reductiontogcd}]
		By squaring out, interchanging summations and employing Poisson summation formula in \cref{lem:Poisson}, we have that
		\begin{align}\label{eq:prooflemma:Poisson1}
			\sum_{\ell\in\mathbb{Z}} \sum_{k\leq 3\alpha N} \frac{\vert R(\alpha\ell+\beta)\vert^2\Phi \parentheses*{\frac{\ell}{N}}}{k^{\sigma+i(\alpha\ell+\beta)}}
			=\sum_{\substack{m,\,n\leq M\\k\leq3\alpha N}}\frac{r(m)r(n)n^{i\beta}}{k^\sigma(mk)^{i\beta}}N\sum_{\ell\in\mathbb{Z}}\widehat{\Phi }\parentheses*{N\parentheses*{\ell-\frac{\alpha}{2\pi}\log\frac{n}{mk}}}.
		\end{align}
		
		Separating the term $\ell=0$ from the inner sum and estimating it, we obtain that
		\begin{align}\label{eq:prooflemma:mainterm}
			\sum_{\substack{m,\,n\leq M\\k\leq3\alpha N}}\frac{r(m)r(n)}{k^\sigma}\parentheses*{\frac{n}{mk}}^{i\beta}N\widehat{\Phi }\parentheses*{N\left(-\frac{\alpha}{2\pi}\log\frac{n}{mk}\right)}=N\sum_{mk=n\leq M}\frac{r(m)r(n)}{k^\sigma}+\mathcal{E}_1,
		\end{align}
		where
		\begin{align*}
			\mathcal{E}_1\ll N\mathop{\sum_{\substack{m,\,n\leq M\\k\leq3\alpha N}}}_{n\neq mk}\frac{r(m)r(n)}{k^\sigma}\parentheses*{1+\frac{\alpha N}{2\pi}\left|\log\frac{n}{mk}\right|}^{-A}.
		\end{align*}
		Observe that $\log\frac{n}{mk}\gg M^{-1}\gg N^{-1/2+\delta}$ for any triple $(k,m,n)$ of indices from the above sum.
		Therefore, by the Cauchy-Schwarz inequality we deduce that
		\begin{align}\label{eq:prooflemma:error1}
			\mathcal{E}_1\hspace*{-2pt}\ll\hspace*{-2pt} 
			N^{2-\sigma}\parentheses*{\sum_{n\leq M}r(n)}^2N^{-A(1/2-\delta)}
			\ll N^{2-\sigma-A(1/2-\delta)}M\sum_{n\leq M}r(n)^2=o(1)\norm{R}_2^2.
		\end{align}
		
		We then proceed on estimating the contribution of the Fourier coefficients with $\ell\neq0$ to the right-hand side in \cref{eq:prooflemma:Poisson1}.
		We set
		\begin{align*}
			\mathcal{E}_2:=\sum_{\substack{m,\,n\leq M\\k\leq3\alpha N}}\frac{r(m)r(n)}{k^\sigma}\parentheses*{\frac{n}{mk}}^{i\beta}N\sum_{\ell\neq0}\widehat{\Phi }\parentheses*{N\left(\ell-\frac{\alpha}{2\pi}\log\frac{n}{mk}\right)}.
		\end{align*}
		Since $\left|\log\frac{n}{mk}\right|\leq(3/2)\log N$, we have
		$\sum_{|\ell|\geq2\log N}\left|\widehat{\Phi }\parentheses*{N\left(\ell-\frac{\alpha}{2\pi}\log\frac{n}{mk}\right)}\right|\ll N^{-A}
		$,
		and thus, similarly as before
		\begin{align}\label{eq:prooflemma:firstreduction}
			\mathcal{E}_2\ll N\sum_{0<|\ell|\leq2\log N}\sum_{\substack{m,\,n\leq M\\k\leq3\alpha N}}\frac{r(m)r(n)}{k^\sigma}\left|\widehat{\Phi }\parentheses*{N\parentheses*{\ell-\frac{\alpha}{2\pi}\log\frac{n}{mk}}}\right|+o(1)\norm{R}_2^2.
		\end{align}
		
		We fix now $0<|\ell|\leq2\log N$ and estimate the sums on $k$, $m$ and $n$. 
		To do so, we introduce an extra parameter $\epsilon=\epsilon(\delta)\in(0,\delta)$ that will be specified in the end.
		Again as above
		\begin{align}\label{eq:prooflemma:secondreduction}
			\begin{split}
				\mathop{\sum_{\substack{m,\,n\leq M\\k\leq3\alpha N}}}_{\left|\ell-\frac{\alpha}{2\pi}\log\frac{n}{mk}\right|>\frac{1}{N^{1-\epsilon}}}\frac{r(m)r(n)}{k^\sigma}\left|\widehat{\Phi }\left(N\left(\ell-\frac{\alpha}{2\pi}\log\frac{n}{mk}\right)\right)\right|
				&\ll N^{1-\sigma-A\epsilon}M\norm{R}_2^2,
			\end{split}
		\end{align} 
		which is also $o(1)\norm{R}_2^2$, and so it remains to estimate
		\begin{align*}
			\mathop{\sum_{\substack{m,\,n\leq M\\k\leq3\alpha N}}}_{\left|\ell-\frac{\alpha}{2\pi}\log\frac{n}{mk}\right|\leq\frac{1}{N^{1-\epsilon}}}\frac{r(m)r(n)}{k^\sigma}.
		\end{align*}
		We split the triple sum into two ranges according to whether $mk<N^{1/2-\epsilon}$ or $mk\geq N^{1/2-\epsilon}$ and we denote the resulting sums by $\Sigma_1(\ell)$ and $\Sigma_2(\ell)$, respectively.

		Any pair of triples $(k,m,n)$ and $(k',m',n')$ in the first range with ${n}{m'k'}\neq{n'}{mk}$ satisfy the relation
		\begin{align*}
			\left|\log\frac{n}{mk}-\log\frac{n'}{m'k'}\right|\geq N^{-1+\delta+\epsilon}.
		\end{align*}
		Also among all pairs of coprime positive integers $a,b\leq N^{1/2-\epsilon}$, there is at most one pair, $(a_\ell,b_\ell)$ say, such that
		\begin{align*}
			\left|\ell-\frac{\alpha}{2\pi}\log\frac{a_\ell}{b_\ell}\right|\leq\frac{1}{N^{1-\epsilon}}.
		\end{align*}
		Therefore a triple $(k,m,n)$ in the first range can only be of the form $n=ja_\ell $, $mk=jb_\ell $ for some positive integer $j\leq M/a_\ell$.
		Hence
		\begin{align*}
			\Sigma_1(\ell)\leq\sum_{j\leq M/a_\ell}r(j a_\ell)\mathop{\sum_{m\leq M}}_{m\mid jb_\ell}r(m)\left(\frac{m}{jb_\ell}\right)^\sigma,
		\end{align*}
		if such a pair $(a_\ell,b_\ell)$ exists.
		Otherwise $\Sigma_1(\ell)=0$.
		Applying \cref{lem:gronwall}, using the fact that $jb_\ell<N^{1/2-\epsilon}$ and applying the Cauchy-Schwarz inequality, we deduce that
		\begin{align}\label{eq:prooflemma:thirdreduction}
			\Sigma_1(\ell)\leq\exp\left(\frac{(1+o(1))\left(\log N\right)^{1-\sigma}}{(1-\sigma)\log\log N}\right)\norm{R}_2^2.
		\end{align}
		
		In the second range $mk\geq N^{1/2-\epsilon}$, the condition $m\leq N^{1/2-\delta}$ implies that $k>N^{\delta-\epsilon}$. 
		Since for any $k,k'\leq 3\alpha N$ with $k\not=k'$ we have
		\begin{align*}
			\left|\frac{\alpha}{2\pi}\log\frac{n}{mk}-\frac{\alpha}{2\pi}\log \frac{n}{mk'}\right|\gg\frac{1}{N}, 
		\end{align*}
		for every pair $(m,n)$ there are at most $O(N^\epsilon)$ possible values of $k$ such that
		\begin{align*}
			\left|\frac{\alpha}{2\pi}\log\frac{n}{mk}-\ell\right|\leq\frac{1}{N^{1-\epsilon}}.
		\end{align*}
		Hence
		\begin{align}\label{eq:prooflemma:fourthreduction}
			\Sigma_2(\ell)\ll N^{\epsilon-\sigma(\delta-\epsilon)} \sum_{m,n\leq M}r(m)r(n)\ll N^{1/2-(1+\sigma)\delta+2\epsilon}\norm{R}_2^2=o(1)\norm{R}_2^2
		\end{align}
		for any $\epsilon>0$ satisfying $2\epsilon<(1+\sigma)\delta-1/2$.
		
		In view of \cref{eq:prooflemma:firstreduction}-\cref{eq:prooflemma:fourthreduction}, we obtain by summing over all $0<|\ell|\leq2\log N$ that
		\begin{align*}
			\mathcal{E}_2\ll N\parentheses*{\exp\left(\frac{(1+o(1))\left(\log N\right)^{1-\sigma}}{(1-\sigma)\log\log N}\right)}\norm{R}_2^2
		\end{align*}
		which, in combination with \cref{eq:prooflemma:Poisson1}-\cref{eq:prooflemma:error1}, yields the proposition.
	\end{proof}
\noindent
	\bibliographystyle{plainnat}
	\bibliography{zetaAP}

\end{document}